%% file: mutual_visibility.tex
\documentclass[12pt]{article}

\usepackage{graphicx,tikz,color,url}
\usepackage{amsmath,amssymb,latexsym}
\usepackage{mathtools}
\usepackage{xspace}

\newtheorem{theorem}{Theorem}[section]
\newtheorem{definition}[theorem]{Definition}
\newtheorem{lemma}[theorem]{Lemma}
\newtheorem{proposition}[theorem]{Proposition}

\newtheorem{corollary}[theorem]{Corollary}

\newcommand{\proof}{\noindent{\bf Proof.\ }}
\newcommand{\qed}{\hfill $\square$ \bigskip}
\newcommand{\cp}{\,\square\,}

\newcommand{\diam}{{\rm diam}}
\newcommand{\gp}{{\rm gp}}
\newcommand{\frog}{{frog\ }}
\newcommand{\mvs}{mutual-visibility set\xspace}

\begin{document}

\title{On the mutual visibility in \\ Cartesian products and triangle-free graphs}

\author{
Serafino Cicerone $^{a}$\thanks{Email: \texttt{serafino.cicerone@univaq.it}}
\and 
Gabriele Di Stefano $^{a}$\thanks{Email: \texttt{gabriele.distefano@univaq.it}} 
\and Sandi Klav\v zar $^{b,c,d}$\thanks{Email: \texttt{sandi.klavzar@fmf.uni-lj.si}}
}
\maketitle

\begin{center}
	$^a$ Department of Information Engineering, Computer Science, and Mathematics, 
	     University of L'Aquila, Italy \\
	\medskip

	$^b$ Faculty of Mathematics and Physics, University of Ljubljana, Slovenia\\
	\medskip

	$^c$ Institute of Mathematics, Physics and Mechanics, Ljubljana, Slovenia\\
	\medskip
	
	$^d$ Faculty of Natural Sciences and Mathematics, University of Maribor, Slovenia\\
	\medskip
\end{center}

\begin{abstract}
Given a graph $G=(V(G), E(G))$ and a set $P\subseteq V(G)$, the following concepts have been recently introduced: $(i)$ two elements of $P$ are \emph{mutually visible} if there is a shortest path between them without further elements of $P$;  $(ii)$ $P$ is a \emph{mutual-visibility set} if its elements are pairwise mutually visible;  $(iii)$ the \emph{mutual-visibility number} of $G$ is the size of any largest mutual-visibility set. 
In this work we continue to investigate about these concepts. We first focus on mutual-visibility in Cartesian products. For this purpose, too, we introduce and investigate independent mutual-visibility sets. In the very special case of the Cartesian product of two complete graphs the problem is shown to be equivalent to the well-known Zarenkiewicz's problem. We also characterize the triangle-free graphs with the mutual-visibility number equal to $3$. 
\end{abstract}



\noindent
AMS Subj.\ Class.\ (2020):  05C12, 05C38, 05C69, 05C76

\section{Introduction}
Let $G = (V(G), E(G))$ be a connected, undirected graph and $X\subseteq V(G)$ a subset of the vertices of $G$. Two vertices of $X$ are \emph{mutually visible} if there exists a shortest path (also called {\em geodesic}) between them without further vertices in $X$. $X$ is a \emph{mutual-visibility set} if its vertices are pairwise mutually visible. If $X\subseteq V(G)$ and $x,y\in X$, then we say that $x$ and $y$ are {\em $X$-visible}, if there exists a shortest $x,y$-path $P$ such that $V(P)\cap X = \{x,y\}$. The size of a largest \mvs is the \emph{mutual-visibility number} of $G$, and it is denoted by $\mu(G)$.

In~\cite{DiStefano-2021+}, the author started the study about this invariant and the {\mvs}s in some classes of graphs, after showing that the problem of finding a \mvs with a size larger than a given number is  NP-complete. 
%
%
There are different motivations for addressing this problem. The first comes from the role that mutual visibility plays in problems arising in the context of distributed computing by mobile entities. Another derives from communication problems in computer/social networks, where vertices of a network in mutual visibility may represent entities that want to efficiently send data in such a way that the exchanged messages do not pass through other entities.

\paragraph{Related work.}
The past two decades have seen rapid growth and development of the field of distributed computing by mobile entities~\cite{flocchini-2012}, whose aim is the study of the computational and complexity issues arising in systems of decentralized computational entities operating either in the Euclidean plane or in some discrete environment (e.g., a graph). In both settings, the research concern is on determining what tasks can be performed by the entities, under what conditions, and at what cost.  One of such basic task is called \textsc{Mutual Visibility} and has been introduced in~\cite{DiLuna17} under the assumption that three collinear entities are not mutually visible: given an arbitrary initial configuration of $n$ entities located at distinct positions, they autonomously arrange themselves in a configuration in which each entity is in a distinct position and from which it can see all other entities. 

%
%
Starting with this initial work, many papers have addressed the same problem (e.g., see~\cite{Aljohani18a,Bhagat20,Poudel2021}). Later, similar visibility problems were considered in different contexts, where the entities are ``fat entities'' modeled as disks in the Euclidean plane (e.g., see~\cite{Poudel19}) or are points on a grid based terrain and their movements are restricted only along grid lines (e.g., see~\cite{Adhikary18}).

However, questions about sets of points and their mutual visibility in the Euclidean plane have been investigated since the end of XIX century. In~\cite{dudeney917} Dudeney posed the still open no-three-in-line problem:  find the maximum number of points that
can be placed in an $n\times n$  grid so that no three points lie on a line. 
Motivated by the Dudeney's problem, in~\cite{manuel:18} the {\sc General Position} problem in graphs was introduced. A couple of years earlier and in a different language, an equivalent problem was posed in~\cite{chand:16}. A subset $S$ of vertices in a graph $G$ is a \emph{general position set} if no triple of vertices from $S$ lie on a common geodesic in $G$. The {\sc General Position} problem is to find a largest general position set of $G$, the order of such a set is the \emph{general position number} $\gp(G)$ of $G$. Since its introduction, the general position number has been studied for several graph classes (e.g., grid networks~\cite{manuel:18b}, cographs and bipartite graphs~\cite{acckt:19},  graph classes with large general position number~\cite{tc:20}, Cartesian products of graphs~\cite{kpry:21, klavzar-2021b, tian-2021a, tian-2021b}, and Kneser graphs~\cite{ghorbani-2021, patkos-2020}).

The difference between a general position set $S$ and a \mvs $X$ is that two vertices are in $X$ if there \emph{exists} a shortest path between them with no further vertex in $X$, whereas two vertices are in $S$ if for \emph{every} shortest path between them no further vertex is in $S$. The two concepts are intrinsically different, but also closely related, since the vertices of a general position set are in mutual visibility.

\paragraph{Outline and results.}
In the rest of the introduction, we provide basic definitions and notation. In the next section, independent mutual-visibility sets are introduced and a couple of related results proved. We use these sets in Section~\ref{sec:cp} to bound the mutual-visibilty number of Cartesian product graphs. Then we prove that if a maximum mutual-visibility set of a graph $G$ is independent, then $\mu(K_k\cp G) = k\cdot \mu(G)$. For the very special case of $\mu(K_m \cp K_n)$ we show that its computation is equivalent to a special case of Zarenkiewicz's problem, which is a notorious open problem.  The main result is given in Section~\ref{sec:triangle-free}, where we prove that if $G$ is a connected, triangle-free graph, then $\mu(G)=3$ if and only if $G$ is a tree with three leaves or a so-called \frog graph.

\paragraph{Basic definitions.}
Since two vertices from different components of a graph are not mutually visible, all graphs in the paper are connected unless stated otherwise. 

For a natural number $n$, we set $[n] = \{1,\ldots, n\}$. Given a graph $G$, $V(G)$ and $E(G)$ are used to denote its vertex set and its edge set, respectively. The order of $G$, that is $|V(G)|$, is denoted by $n(G)$. The distance function $d_G$ on a graph $G$ is the usual shortest-path distance. The {\em diameter} $\diam(G)$ of $G$ is the maximum distance between pairs of vertices of the graph.

If $X\subseteq V(G)$, then $G[X]$ denotes the subgraph of $G$ induced by $X$. We denote by $P_k$ any induced path with $k\ge 1$ vertices, by $C_k$ any chordless cycle with $k\ge 3$ vertices, and by $K_k$ the clique with $k\ge 1$ vertices. Given a cycle $C$, two vertices $u,v$ of $C$ are \emph{antipodal} if $d_C(u,v)=\left\lfloor\frac{n(C)}{2}\right\rfloor$.

An independent set is a set of vertices of $G$, no two of which are adjacent. The cardinality of a largest independent set is the {\em independence number} $\alpha(G)$ of $G$. 

A subgraph $G'$ of a graph $G$ is \emph{isometric}, if for every two vertices of $G'$, the distance between them in $G'$ equals the distance in $G$. The subgraph $G'$ is \emph{convex}, if for every two vertices of $G'$, every shortest path in $G$ between them lies completely in $G'$. Clearly, each convex subgraph is isometric.

\section{Independent mutual-visibility number}
\label{sec:mu-i}

In this section we introduce independent mutual-visibility sets. They will turn out to be useful in the next section, but we believe that they may also be of independent interest. 

Let $X$ be a mutual-visibility set of a graph $G$. If $G[X]$ is edgeless, then we say that $X$ is an {\em independent mutual-visibility set}. The {\em independent mutual-visibility number} $\mu_i(G)$ of $G$ is the size of a largest independent mutual-visibility set. Clearly, if $G$ is an arbitrary graph, then 
\begin{equation}
\label{eq:upper-mu-alpha}
\mu_i(G) \le  \min \{\mu(G), \alpha(G)\}\,.
\end{equation}
If $G$ has small diameter, then we have the equality, more precisely, the following holds.  

\begin{lemma}
\label{lem-mu-i-diam-3}
If $\diam(G)\le 3$, then $\mu_i(G) = \alpha(G)$. 
\end{lemma}

\proof
If $\diam(G) = 1$, then $G$ is a complete graph and hence $\mu_i(G) = 1 = \alpha(G)$. Suppose next that $\diam(G)\in \{2,3\}$ and let $X$ be an arbitrary independent set of $G$. We claim that $X$ is also an independent mutual-visibility set. Let $x, y$ be arbitrary vertices from $X$ and let $P$ be a shortest $x,y$-path. Since $X$ is an independent set, $P$ is of length $2$ or $3$. But then, again using the fact that $X$ is independent, $V(P)\cap X = \{x,y\}$. This means that $x$ and $y$ are $X$-visible and consequently $\mu_i(G) \ge \alpha(G)$. The argument is completed by applying~\eqref{eq:upper-mu-alpha}. 
\qed

Lemma~\ref{lem-mu-i-diam-3} need not hold for graphs $G$ with $\diam(G)\ge 4$. For instance, $\mu_i(P_5) = 2$ and $\alpha(P_5) = 3$. 

In view of Lemma~\ref{lem-mu-i-diam-3} we are interested in the graphs $G$ for which $\mu(G) = \mu_i(G)$ holds. Additional reasons for this interest will be given in Section~\ref{sec:cp}. 

As the first example note that $\mu(C_k) = 3 =  \mu_i(C_k)$ holds for each $k\ge 6$. From~\cite[Corollary 4.3]{DiStefano-2021+} we know that if $L$ is the set of leaves of a tree $T$, then $L$ is a mutual-visibility set and $\mu(T) = |L|$. Hence, if $n(T)\ge 3$, then 
\begin{equation}
\label{eq:trees}
\mu(T) = |L| = \mu_i(T)\,.
\end{equation}

The {\em corona} $G\circ H$ of disjoint graphs $G$ and $H$ was introduced in~\cite{Frucht-1970} as the graph obtained from the disjoint union of $G$ and $n(G)$  copies of $H$ by joining the $i^{\rm th}$ vertex of $G$ to every vertex in the $i^{\rm th}$ copy of $H$. This graph operation has so far been explored from many aspects, see~\cite{Dong-2021, Furmanczyk-2021, Kuziak-2017}.  Here we add the following result which in particular yields a large class of graphs with $\mu = \mu_i$.

\begin{proposition}
\label{prop:corona}
If $G$ is a (connected) graph with $n(G)\ge 2$, and $H$ is an arbitrary graph, then 
$$\mu(G\circ H) = n(G)n(H)\,.$$
Moreover, for $k\geq 1$, $\mu(G\circ \overline{K}_k) = k\cdot n(G) = \mu_i(G\circ \overline{K}_k)$.
\end{proposition}

\proof
Let $V(G) = \{v_1, \ldots, v_{n(G)}\}$, and let $H_i$ be the copy of $H$ whose vertices are in $G\circ H$  adjacent to $v_i$, $i\in [n(G)]$. Then $(G\circ H)[V(H_i)\cup \{v_i\}]$ is the join of the one vertex graph $K_1$ and the graph $H_i$. 
If $H_i$ is not complete, we can see it directly or deduce from~\cite[Corollary 4.10]{DiStefano-2021+} that 
%
%
$\mu((G\circ H)[V(H_i)]) = n(H)$. Hence, no matter whether $H$ is complete or not, $\cup_{i=1}^{n(G)} V(H_i)$ is a mutual-visibility set of $G\circ H$ and consequently $\mu(G\circ H) \ge n(G)n(H)$. 

For the reversed inequality assume on the contrary that $\mu(G\circ H) > n(G)n(H)$. Let $X$ be an arbitrary mutual-visibility set of $G\circ H$ of order $\mu(G\circ H)$. Then $X$ necessarily contains at least one vertex of $G$. That is, setting $k = |X\cap V(G)|$, we have $k\ge 1$. Note that if $v_i\in X$, then, since $n(G)\ge 2$, we have $V(H_i)\cap X = \emptyset$. It follows that $|X| \le k + (n(G) - k)n(H) = n(G)n(H) - k(n(H)-1) \le n(G)n(H)$. This contradiction implies that $\mu(G\circ H) \le n(G)n(H)$.

The second assertion of the theorem follows from the already proved assertion and the fact that the set $\cup_{i=1}^{n(G)} V(H_i)$ is independent when $H = \overline{K}_k$. 
\qed

We point out that in Proposition~\ref{prop:corona} the graph $H$ need not be connected because $G\circ H$ is connected no matter whether $H$ is connected or not.

\section{Mutual-visibility in Cartesian products}
\label{sec:cp}

In this section we first give an upper and a lower bound on the mutual-visibility number of general Cartesian product graphs, where the latter bound is expressed in terms of the independent mutual-visibility number. Then we prove that if $\mu_i(G) = \mu(G)$, then $\mu(K_k\cp G) = k\cdot \mu(G)$. In the second part we reduce the problem of computing $\mu(K_m \cp K_n)$ to the well-known Zarenkiewitz's problem~\cite{jukna-2011, nagy-2019, west-2021, z-1951}. This implies that to compute $\mu(G\cp H)$ is an intrinsically difficult problem. 

Recall that the {\em Cartesian product} $G\cp H$ of graphs $G$ and $H$ has the vertex set $V(G)\times V(H)$ and the edge set $E(G\cp H)  = \{(g,h)(g',h'):\ gg'\in E(G)\mbox{ and } h=h', \mbox{ or, } g=g' \mbox{ and }  hh'\in E(H)\}$. If $(g,h)\in V(G\cp H)$, then the {\em $G$-layer}  $G^h$ through the vertex $(g,h)$ is the subgraph of $G\cp H$ induced by the vertices $\{(g',h):\ g'\in V(G)\}$. Similarly, the {\em $H$-layer} $^g\!H$ through $(g,h)$ is the subgraph of $G\cp H$ induced by the vertices $\{(g,h'):\ h'\in V(H)\}$. Recall further the following known result. (For more information on the Cartesian product we refer to the book~\cite{Hammack-2011}.)

\begin{lemma}
\label{lem:distance-in-cp}
If $G$ and $H$ are graphs and $(g,h),(g',h')\in V(G\cp H)$, then 
$$d_{G \cp H}((g,h),(g',h')) = d_{G}(g,g') + \,d_{H}(h,h')\,.$$
\end{lemma}

Our first result on Cartesian products reads as follows. 

\begin{theorem}
\label{thm:cp-bounds}
If $G$ and $H$ are graphs, then the following holds. 
\begin{align*}
(i) & \quad \mu(G\cp H) \le \min \{\mu(G)n(H), \mu(H)n(G)\}\,, \\
(ii) & \quad \mu(G\cp H) \ge \max \{\mu(G)\mu_i(H), \mu(H)\mu_i(G)\}\,.
\end{align*}
\end{theorem}

\proof
(i) If $h\in V(H)$, then the $G$-layer $G^h$ is a convex subgraph of $G\cp H$, hence each pair of vertices from $V(G^h)$ can only be mutually visible inside the layer. Applying this fact to each vertex of $H$ we get that $\mu(G\cp H) \le n(H)\mu(G)$, cf.~\cite[Lemma~2.3]{DiStefano-2021+}. Analogously we see that $\mu(G\cp H) \le n(G)\mu(H)$ also holds. 

(ii) Let $X_G$ be a mutual-visibility set of $G$ of size $\mu(G)$, and let $X_H$ be an independent mutual-visibility set of $H$ of size $\mu_i(H)$. We claim that $X = X_G \times X_H$ is a mutual-visibility set of $G\cp H$. 

Let $x = (g,h)$ and $y = (g',h')$ be different vertices of $X$. We need to show $x$ and $y$ are $X$-visible in $G\cp H$. Suppose first that $h = h'$. Then $x$ and $y$ both lie in the $G$-layer $G^h$. Since $G^h$ is a convex subgraph of $G\cp H$ isomorphic to $G$, and $X_G$ is a mutual-visibility set of $G$, the vertices $x$ and $y$ are $X$-visible in $G\cp H$. In the rest we may thus assume that $h\ne h'$. 

Since  $X_G$ is a mutual-visibility set of $G$, there exists a shortest $g,g'$-path $P_G$ such that $V(P_G) \cap X_G = \{g,g'\}$. (In the case when $g=g'$, the path $P_G$ consists of a single vertex $g=g'$.) Further,  since $X_H$ is a mutual-visibility set of $H$, there exists a shortest $h,h'$-path $P_H$ such that $V(P_H) \cap X_H = \{h,h'\}$. Let $h''$ be the neighbor of $h$ on $P_H$. Since $X_H$ is independent we infer that $h'' \ne h'$. Consider now the $x,y$-path $P_{G\cp H}$ defined as follows: $P_{G\cp H}$ starts with the edge $(g,h)(g,h'')$, then continues along the copy of $P_G$ in the layer $G^{h''}$, and ends with the copy of the $h'',h'$-subpath of $P_H$ in the layer $^{g'}\!H$. (Note that if $g=g'$, then $P_{G\cp H}$ is simply the copy of $P_H$ in the layer $^{g}\!H$.) Using Lemma~\ref{lem:distance-in-cp} it follows that $P_{G\cp H}$ is a shortest $x,y$-path in $G\cp H$. Moreover, since  $V(P_{G\cp H}) \cap X = \{x,y\}$ we conclude that  $X$ is a mutual-visibility set. Hence $\mu(G\cp H) \ge \mu(G)\mu_i(H)$. By the commutativity of the Cartesian product we also have $\mu(G\cp H) \ge \mu(H)\mu_i(G)$ and the result follows. 
\qed

Theorem~\ref{thm:cp-bounds}(ii) together with~\eqref{eq:trees} yields the following consequence which should be compared with the main theorem of~\cite{tian-2021b} asserting that $\gp(T_1\cp T_2)  =  \gp(T_1) + \gp(T_2)$ holds for trees $T_1$ and $T_2$ of order at least $3$. 

\begin{corollary}
\label{cor:trees-lower}
If  $T_1$ and $T_2$ are trees of order at least $3$, then $\mu(T_1\cp T_2) \ge \mu(T_1)\mu(T_2)$.
\end{corollary} 

The bound of Theorem~\ref{thm:cp-bounds}(i) is sharp. For instance, it was proved in~\cite[Theorem 4.6]{DiStefano-2021+} that if $ r >3$ and $s > 3$,  then $\mu(P_r\cp P_s) = 2\cdot \min \{r,s\}$. Moreover, the two bounds of  Theorem~\ref{thm:cp-bounds} can coincide. We demonstrate this by the next result in which we determine the mutual-visibility number of a large class of Cartesian product graphs.  

\begin{theorem}
\label{thm:cp-complete-factor}
If $k\ge 2$ and  $G$ is a graph with $\mu(G) = \mu_i(G)$, then 
$$\mu(K_k\cp G) = k\cdot \mu(G)\,.$$ 
\end{theorem}

\proof
Since $\mu(K_k) = k$, Theorem~\ref{thm:cp-bounds}(i) yields 
$$\mu(K_k\cp G) \le n(K_k) \cdot \mu(G) = k\cdot \mu(G)\,.$$
On the other hand, using Theorem~\ref{thm:cp-bounds}(ii) and the assumption $\mu(G) = \mu_i(G)$, we get 
$$\mu(K_k\cp G) \ge \mu(K_k) \cdot \mu_i(G) =  k\cdot \mu(G)$$
and we are done. 
\qed

\subsection{Reduction of $\mu(K_m\cp K_n)$ to Zarankievicz's problem}

The following lemma is the key for the reduction from the title. 

\begin{lemma}
\label{lem:hamming-graphs}
Let $r,s\ge 2$ and let $X\subseteq V(K_r\cp K_s)$. Then $X$ is a mutual-visibility set of $K_r\cp K_s$ if and only if $|X\cap V(C)|\le 3$ holds for each induced $4$-cycle $C$ of $K_r\cp K_s$. 
\end{lemma}

\proof
Let $V(K_k) = [k]$ and set $G = K_r\cp K_s$. Suppose first that $X$ is a mutual-visibility set of $G$. Consider an arbitrary induced $4$-cycle $C$ of $G$. Then $C$ has consecutive vertices of the form $(i,k), (j,k), (j,\ell), (i,\ell)$, where $i\ne j$ and $k\ne \ell$. In particular, $C$ is a convex subgraph of $G$ and hence $|X\cap V(C_4)|\le 3$ must hold. 

Conversely, let $X\subseteq V(K_r\cp K_s)$ be such that $|X\cap V(C)|\le 3$ holds for each induced $4$-cycle $C$ of $G$. Consider arbitrary vertices $x = (i,k)$ and $y = (j,\ell)$ of $X$. If either $i=j$ or $k=\ell$, then $x$ and $y$ lie either in a common $K_s$-layer or in a common $K_r$-layer. In either case, $x$ and $y$ are adjacent and hence $X$-visible. Suppose next that $i\ne j$ and $k\ne \ell$. Then the vertices $x$ and $y$ lie in the convex $4$-cycle $C$ with the consecutive vertices $(i,k), (j,k), (j,\ell), (i,\ell)$. Since $|X\cap C| \le 3$, we may without loss of generality assume that $(j,k)\notin X$. But then the path $x=(i,k), (j,k), (j,\ell) = y$ demonstrates that $x$ and $y$ are $X$-visible. Hence we can conclude that  $X$ is a mutual-visibility set of $G$.
\qed

Each induced $4$-cycle $C$ of $K_r\cp K_s$ can be written as $K_2\cp K'_2$, where $K_2$ is the projection of $C$ onto $G$ and $K_2'$ is the projection of $C$ onto $H$. One often says that such a $4$-cycle is a {\em Cartesian square}. With this terminology in hand, Lemma~\ref{lem:hamming-graphs} asserts that the problem of computing $\mu(K_r\cp K_s)$ is to find a largest subset of vertices $X$, so that $X$ intersects each Cartesian square in at most three vertices. This brings us to the following: 

\begin{definition}[Zarankiewicz's Problem]
Let $m$, $n$, $s$, and $t$ be given positive integers. Determine $z(m, n; s, t)$ the maximum number of $1$s that an $m\times n$ binary matrix can have provided that it contains no constant $s\times t$ submatrix of $1$s.
\end{definition}

From Lemma~\ref{lem:hamming-graphs} we deduce: 

\begin{corollary}
\label{cor:we-discovered-Zarenkiewitz}
If $m, n\ge 2$, then $\mu(K_m\cp K_n) = z(m,n;2,2)$. 
\end{corollary}

Values $z(m,n;s,t)$ have been extensively studied, cf.~\cite{west-2021}, but even the values $z(m,n;2,2)$ are not known. We recall the following most relevant results that can be found in~\cite{west-2021}, the first two as Theorems 13.2.19. and Theorem 13.2.21.

\begin{theorem}
\label{thm:kovari}
{\rm \cite[K\H{o}v\'ari–-S\'os-Tur\'an, 1954]{kovari-1954}}
For $s,t>1$,     
$$z(m,n;s,t) < (s-1)^{1/t} (n-t+1)m^{1-1/t} + (t-1)m\,.$$
\end{theorem}
Corollary~\ref{cor:we-discovered-Zarenkiewitz} together with Theorem~\ref{thm:kovari} for the case $s=t=2$ yields:
$$\mu(K_m\cp K_n) = z(m,n;2,2) < (n-1) \sqrt{m} + m = (n-1) \sqrt{m} +m\,.$$

\begin{theorem}
\label{thm:erdos}
{\rm \cite[Brown, 1966; Erd\H{o}s-R\'enyi-S\'os, 1966]{brown-1966, erdos-1966}}
When n is sufficiently large,
$$n^{3/2}-n^{4/3} \le z(n,n;2,2) \le \frac{1}{4} n(1+ \sqrt{4n-3})\,.$$ 
\end{theorem}

\begin{theorem}
\label{thm:west}
{\rm \cite[Theorem 13.2.24.]{west-2021}}
If  $\alpha = \frac{s-1}{st-1}$ and $\beta = \frac{t-1}{st-1}$,
then
$$z(m,n;s,t) \ge \left\lfloor \left(1 - \frac{1}{s!\,t!}\right) m^{1-\alpha} n^{1-\beta} \right\rfloor\,.$$ 
\end{theorem}
We note in passing that the usual construction leading to this lower bound(s) is to use projective planes. In the special case of interest to us, Corollary~\ref{cor:we-discovered-Zarenkiewitz} together with Theorem~\ref{thm:west} yields: 
$$\mu(K_m\cp K_n) = z(m,n;2,2) \ge \left\lfloor \frac{3}{4} m^{2/3} n^{2/3} \right\rfloor\,.$$

\section{Triangle-free graphs $G$ with $\mu(G)=3$}
\label{sec:triangle-free}

In~\cite{DiStefano-2021+}, the following characterizations are proved:
\begin{itemize}
\item $\mu(G)=1$ if and only if $G\simeq K_1$; 
\item $\mu(G)=2$ if and only if $G\simeq P_n$, $n\ge 2$. 
\end{itemize}
In this section we supplement these results by characterizing triangle-free graphs $G$ with $\mu(G) = 3$. To this end, some preparation is needed. 
 
\begin{lemma}
\label{lem:isometric}
 If $G'$ is an isometric subgraph of $G$, then $\mu(G)\geq \mu(G')$.
\end{lemma}

\begin{proof}
Let $X$ be a mutual-visibility set of $G'$. For each $u,v\in X$ there is a shortest $u,v$-path $P$ in $G'$ such that $X\cap V(P)=\{u,v\}$. Since $P$ is a shortest $u,v$-path also in $G$, then $u,v$ are mutually visible also in $G$. Hence $X$ is also a mutual-visibility set of $G$ and thus $\mu(G)\geq \mu(G')$.
\qed
\end{proof}

Since a convex subgraph is an isometric, Lemma~\ref{lem:isometric} extends~\cite[Lemma~2.2]{DiStefano-2021+} which is stated for convex subgraphs. We also need the following fact based on $\Delta(G)$, the maximum degree of a vertex in a graph $G$. 

\begin{lemma}{\rm \cite[Lemma~2.6]{DiStefano-2021+}}
\label{lem:Delta-bound}
If $G$ is a graph, then $\mu(G) \ge \Delta(G)$. 
\end{lemma}

A \emph{\frog graph} is obtained by composing a cycle $C$ with two paths 
$P_r$ and $P_s$, $r,s\ge 1$ (cf. Fig.~\ref{fig:frog_1}) as follows: given two antipodal vertices $v^r$ and $v^s$ of $C$, identify $v^r$ ($v^s$, respectively) with an external vertex of $P_r$ ($P_s$, respectively). See Fig.~\ref{fig:frog_1}. Vertices $v^r$ and $v^s$ are called \emph{junction vertices} of the \frog graph. Notice that each cycle graph is also a frog graph (it occurs when $r=s=1$).

\begin{figure}[ht!]
\begin{center}
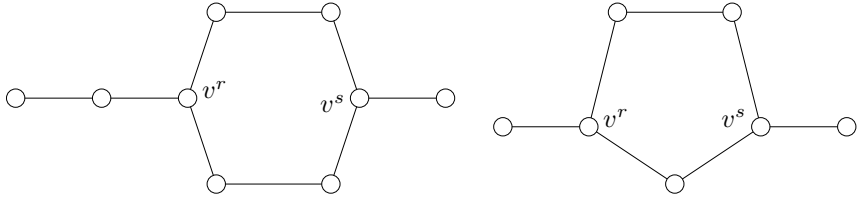
\end{center}
\caption{%
Examples of frog graphs.
}
\label{fig:frog_1}
\end{figure}

The {\em graph $H$} is the graph obtained from the disjoint union of two paths on three vertices by adding an edge between the central vertices of the two paths, see Fig.~\ref{fig:H}. 

\begin{figure}[ht!]
\begin{center}
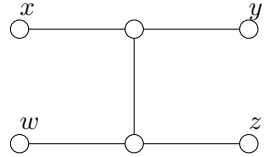
\end{center}
\caption{\small
The graph $H$.
}
\label{fig:H}
\end{figure}

\begin{lemma}
\label{lem:h-graph}
If the graph $H$ is a subgraph of a triangle-free graph $G$, then $\mu(G)\geq 4$.
\end{lemma}

\begin{proof}
Consider a subgraph $H$ of $G$ and let $X=\{x,y,w,z\}$ be the set containing the four pendant vertices of $H$ as described in Fig.~\ref{fig:H}. 

We claim that $X$ is a mutual-visibility set of $G$. As $G$ is triangle-free, the $x,y$-path of length $2$ in $H$ is a shortest path in $G$ and hence $x$ and $y$ are $X$-visible in $G$. Symmetrically, $w$ and $z$ are $X$-visible in $G$. It remains to show that each pair of vertices at distance $3$ in $H$ is $X$-visible in $G$. It suffices to consider the pair $x,z$ because the argument for the other three such pairs is analogous. Clearly, $d_G(x,z) \le d_H(x,z) = 3$. There is nothing to prove if $d_G(x,z) = 1$. Suppose  $d_G(x,z) = 2$. If $x$ and $z$ are not $X$-visible in $G$, each common neighbor of $x$ and $z$ must be from $\{y,w\}$,  but this is not possible as $G$ is triangle-free. Finally, if $d_G(x,z) = 3$, then $x$ and $z$ are clearly $X$-visible in $G$.  This proves the claim which in turn implies that $\mu(G)\ge 4$.
 \qed
\end{proof}

All is now ready for our main result. 

\begin{theorem}\label{thm:mu3}
Let $G$ be a connected, triangle-free graph. Then $\mu(G)=3$ if and only if $G$ is a tree with three leaves or a \frog graph.
\end{theorem}

\begin{proof}
($\Leftarrow$) If $G$ is a tree with three leaves, then $\mu(G)=3$ as proved in~\cite[Corollary 4.3]{DiStefano-2021+}. If $G$ is a \frog graph, then let $C$ be its cycle, $P_r$ and $P_s$ its attached paths, and $v^r,v^s$ the corresponding junction vertices. Since $C$ is a convex subgraph of $G$ (and hence also isometric), Lemma~\ref{lem:isometric} gives $\mu(G)\ge \mu(C)$. This implies $\mu(G)\geq 3$ since $\mu(C)=3$ for any cycle (cf.~\cite[Lemma 2.8]{DiStefano-2021+}). Assume now that there exists in $G$ a mutual-visibility set $X$ with $|X|\ge 4$. Notice first that $|P_r\cap X|\le 1$ and $|P_s\cap X|\le 1$ must hold. Indeed, if, say, $P_r\cap X = \{u_1,u_2\}$ and $u_3\in (P_s\cup C)\cap X$, then either $u_3$ and $u_1$ or $u_3$ and $u_2$ are not $X$-visible. Notice also that $X$ cannot be completely contained in $C$. If $P_r\cap X = \{u_1\}$, let $u_2,u_3\in X$ be the vertices of $C$ closest to $v^r$. Consider the two $v^r,v^s$-paths of $C$: $u_2$ and $u_3$ must belong to the same path, otherwise $u_1$ and any additional vertex $u_4\in X$ are not $X$-visible. But, if $u_2$ and $u_3$ belong to the same path, then since $v^r$ and $v^s$ are antipodal, $u_1,u_2,u_3$ are not pairwise $X$-visible.  Then $\mu(G)\le 3$.
 
($\Rightarrow$) Since $\mu(G) = 3$, Lemma~\ref{lem:Delta-bound} implies that $\Delta(G)\le 3$. If $\Delta(G) = 2$, then $G$ cannot be a path as we would have $\mu(G)\le 2$. 
In the rest we may assume that $\Delta(G) = 3$. We analyze different cases according to the number of cycles in $G$.

\smallskip
Assume that $G$ has no cycles. Then $G$ is a tree with $\Delta(G) = 3$ and $\mu(G) = 3$. Moreover, $G$ must have exactly three leaves, for otherwise $\mu(G)\ge 4$ would hold.  
 
\smallskip
Assume that $G$ has a single cycle $C$. We show that $G$ is a \frog graph. Recall that $\Delta(G) = 3$ and consider first the case in which $G$ contains two vertices $u,v$ with degree $3$ and at least one of them (say $u$) is not on $C$. 
In case $v$ is on $C$, assume without loss of generality that $v$ is the vertex of $C$ with degree 3 closest to $u$.
Then, the two neighbors of $u$ and the two neighbors of $v$ which do not lie on a shortest $u,v$-path, form a mutual-visibility set of $G$, a contradiction with $\mu(G)=3$. Hence, if in $G$ there are vertices with degree $3$, they all are on $C$. Consider the case in which there are at least three  such vertices. Then at least two of them (say $u,v$) are not antipodal and hence there exists a mutual-visibility set of $G$ containing the four vertices adjacent to $u$ and $v$ not belonging to the shortest $u,v$-path. Again a contradiction with $\mu(G)=3$. Two cases remain: (1) if $C$ has a single vertex $u$ such that $\deg(u)= 3$, then $G$ is a \frog graph; (2) if $C$ has two vertices $u,v$ with degree $3$, they must be antipodal (otherwise we get the previous contradiction) and hence $G$ is again a \frog graph. 

\smallskip
Assume finally that $G$ has more than one cycle. To complete the proof of the theorem we are going to show that $\mu(G)\geq 4$. In the first subcase suppose that $G$ contains two vertex-disjoint cycles $C'$ and $C''$. Let $u'\in V(C')$ and $u''\in V(C'')$ be selected such that $d_G(u',u'')$ is smallest possible. Then, having in mind that $G$ is triangle-free, the four vertices in $(N(u')\cap C')\cup(N(u'') \cap C'')$ are pairwise mutual-visible. Observe next that if $C'$ and $C''$ are cycles of $G$, then $|V(C')\cap V(C'')| = 1$ is not possible because $\Delta(G) = 3$. If two cycles share one single edge $uv$, then $G$ contains $H$ and Lemma~\ref{lem:h-graph} yields $\mu(G)\ge 4$. 

It remains to analyze the case in which  each pair of cycles of $G$ shares more than one edge. Select a pair of cycles $C'$, $C''$ such that $n(C') + n(C'')$ is as small as possible, and assume without loss of generality that $n(C')\leq n(C'')$. Then, both $C'$ and $C''$ are isometric subgraphs of $G$.  According to the minimality of $n(C') + n(C'')$, it follows that $C'\cap C'' = P_n$, where $n\ge 3$ by the case assumption. 

Consider the subgraph $F=G[V(C')\cup V(C'')]$ and let $u,v$ be the end-vertices of $P_n$. Then $\deg_F(u) = \deg_F(v) = 3$ and $u$ and $v$ are connected by three paths in $F$. We may without loss of generality assume that $P_n$ is a shortest among these three $u,v$-paths. Let $P_l$ and $P_m$ be the other two paths such that $l\geq m\geq n$. Then $C'=G[V(P_n)\cup V(P_m)]$ and $C''=G[V(P_n)\cup V(P_l)]$. 

Let $x,y,z$ be three vertices such that:
\begin{itemize}
\item
$P_l= (u, \ldots, z, \ldots, v)$, where
$z$ divides $P_l$ in the $u,z$-subpath of length $d\ge 1$ and the $z,v$-subpath of length $l-d-1$; 
\item
$P_n= (u, x, \ldots, v)$, where
$x$ divides $P_n$ in the $u,x$-subpath of length $1$ and the $x,v$-subpath of length $n-2$; 
\item
$P_m= (u, \ldots, y, \ldots, v)$, where
$y$ divides $P_m$ in the $u,y$-subpath of length $d'\ge 1$ and the $y,v$-subpath of length $m-d'-1$. 
\end{itemize}

\begin{figure}[ht!]
\begin{center}
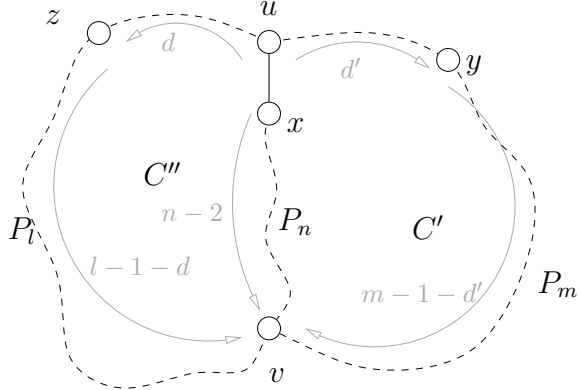
\end{center}
\caption{\small
The subgraph $F$ as defined in the proof of Theorem~\ref{thm:mu3}. The solid line represents an edge, the dashed lines represent paths, and the gray arrows show the lengths of some subpaths.
}
\label{fig:mu3}
\end{figure}

We set now a condition for the position of $z$ on $P_l$ such that $x,v,z$ are mutually visible in $C''$. To this aim, the following conditions must hold:
\begin{itemize}
\item  $(l-1-d) \le (1+d)+ (n-2)$;
\item  $(1+d)   \le (l-1-d) + (n-2)$;       
\item  $(n-2)   \le (l-1-d) + (1+d)$.
\end{itemize}
The last condition corresponds to $l \ge n-2$, which holds by hypothesis. The first two conditions yield the following constraints about $d$:
\begin{equation}\label{eq:2d}
l-n \le 2d \le l+n-4\,.
\end{equation}
This implies that when the length $d$ of the $u,z$-subpath fulfills \eqref{eq:2d}, $x,v,z$ are mutually visible in $C''$. Analogously in $C'$, by letting $d'$ fulfilling the following constraints,
\begin{equation}\label{eq:2d'}
m-n \le 2d' \le m+n-4,
\end{equation}
then we get that $x,v,y$ are mutually visible in $C'$.


\smallskip
We now check whether \eqref{eq:2d} and~\eqref{eq:2d'} are sufficient to have that $X=\{v,x,y,z\}$ is a mutually visible set of $F$. To ensure this property, we need to prove the following: (i) $y$ does not obstruct the visibility of $z$ and $v$; (ii) $z$ does not obstruct the visibility of $y$ and $v$; 
(iii) both $x$ and $v$ do not obstruct the visibity of $y$ and $z$, that is, the $y,z$-path passing through $u$ and with length $d+d'$ is a shortest path in $F$.

Concerning proving (i), it corresponds to show that $(l-1-d) \leq d + d' + (m-1-d')$. The latter is equivalent to $2d \ge l-m$, which is guaranteed by~\eqref{eq:2d}. Similarly, (ii) corresponds to show that $ (m-1-d') \leq d + d' + (l-1-d)$; this is equivalent to $2d'\ge m-l$ which trivially holds since $l\ge m\ge n$.

Concerning proving (iii), it corresponds to show that all the following relationships hold:
\begin{itemize}
\item  $d+d' \le (m-1-d') + (n-2) + (1+d)$. 
This relationship holds by~\eqref{eq:2d'}.
\item  $d+d' \le (d'+1) + (n-2) + (l-1-d)$. 
This relationship holds by~\eqref{eq:2d}.
\item  $d+d' \le (m-1-d') + (l-1-d)$. 
It corresponds to $2d+2d' \le m +l -2 = (m+n-4) + (l-n+2)$. Since $2d' \le  m+n-4$ by~\eqref{eq:2d'}, if we overwrite~\eqref{eq:2d} by further imposing
\begin{equation}\label{eq:2d-new}
l-n \le 2d \le l-n+2\,,
\end{equation}
then we finally get that $X=\{v,x,y,z\}$ is a mutually visible set of $F$. 
\end{itemize}
In the rest of the proof we select $d$ and $d'$ such that they fulfill~\eqref{eq:2d'} and~\eqref{eq:2d-new}, respectively, and fix them. In this way the vertices $z$ and $y$ are uniquely defined.

\smallskip
Assume now that $X$ is not a mutual-visibility set of $G$. The only possibility for this is due to some path $P_t$ that connects a vertex in $P_l$ to a vertex in $P_m$. It can be easily observed that this occurs when $P_t$ connects $z$ and $y$ and two vertices in $X$ are not mutually visible, that is when one of the following pairs is not mutually visible:
\begin{itemize}
\item[$(a)$]  $y$ and $v$ (that is, $z$ obstructs the visibility of $y$ and $v$);
\item[$(b)$]  $z$ and $v$ (that is, $y$ obstructs the visibility of $z$ and $v$);
\item[$(c)$]  $z$ and $x$ (that is, $y$ obstructs the visibility of $z$ and $x$); 
\item[$(d)$]  $y$ and $x$ (that is, $z$ obstructs the visibility of $y$ and $x$).
\end{itemize}
Assume that $P_t$ has no vertex in common with the $y,z$-path passing on $u$ in $F$. We now prove that, in each of the above cases, the cycle $C'''$ formed by $P_t$, by the $y,u$-subpath of $P_m$ of length $d'$, and by the $u,z$-subpath of $P_l$ of length $d$ has order less than $n(C'')$, that is
$ n(C''')= (t-1) + d + d' < n(C'')= l+n-2 < l+n $,
which is a contradiction with the minimality of $n(C') + n(C'')$. Notice that the latter is equivalent to 
\begin{equation}\label{eq:t2}
t <  l+n - d - d' + 1.
\end{equation}


Consider the case $(a)$, in which $z$ obstructs the visibility of $y$ and $v$. In this case, the length $t-1$ of the $y,z$-path $P_t$ must guarantee $(t-1) + (l-1-d) < m-1-d'$, that is:
\begin{equation}\label{eq:t1}
t < m +d -d' - l + 1.
\end{equation}
According to~\eqref{eq:t1}, to prove~\eqref{eq:t2} it is enough to show that the following holds:
\[ l+n - d - d' + 1 \ge m +d -d' - l + 1. \]
Notice that this inequality is equivalent to $2d \le n+2l-m$,
which trivially holds by using~\eqref{eq:2d-new}. 

Consider $(b)$. In this case, the length $t-1$ of the $y,z$-path $P_t$ must guarantee $(t-1) + (m-1-d')  < l-1-d$. By using the same approach as in case $(a)$, this relationship still leads to prove that $n(C''') < n(C'')$. In fact, it is equivalent to showing that $2d' \le m+n$, which is trivially implied by~\eqref{eq:2d'}.

Consider $(c)$. In this case, the length $t-1$ of the $y,z$-path $P_t$ must guarantee $(t-1) + d' +1  < d+1$. Proving $n(C''') < n(C'')$ is equivalent to showing that $2d \le l+n$, which is trivially implied by~\eqref{eq:2d-new}.

Consider $(d)$. In this case, the length $t-1$ of the $y,z$-path $P_t$ must guarantee $(t-1) + d +1  < d'+1$. Proving $n(C''') < n(C'')$ is  equivalent to showing that $2d' \le l+n$, which is implied by~\eqref{eq:2d'} and by the assumption $l\ge m$.

This proves the claim that $X$ is a mutual-visibility set of $G$ when $P_t$ has no vertex in common with the $y,z$-path $P$ passing on $u$ in $F$. Assume now that $P_t$ shares some vertex with such a path $P$. Since $P_t$ connects $y$ and $z$, it must form at least a cycle with $P$ of length less than $n(C''')$. As above, this contradicts the minimality of $n(C') + n(C'')$.  
\qed
\end{proof}

\section{Concluding remarks}

As a continuation of the present investigation, it would in particular be desirable to characterize the trees for which the equality holds in Corollary~\ref{cor:trees-lower}, and to characterize all graphs $G$ with $\mu(G)=3$. 

As pointed out in the introduction, mutual-visibility sets are conceptually similar to general position sets. Hence it would be interesting to investigate graphs $G$ with $\gp(G) = \mu(G)$, where $\gp(G)$ is the general position number of $G$. For instance, the equality holds in geodetic graphs, that is, graphs with the property that each pair of vertices is connected by a unique shortest path, cf.~\cite{mao-1999, parthasarathy-1984, stemple-1968}. 

Note that the Petersen graph is geodetic. In this respect it would be of interest to consider 
$\mu(G)$ for graph $G$ with $\diam(G) = 2$.

\end{document}

%% file: frog_1.pdf_t
\begin{picture}(0,0)%
\includegraphics{frog_1.pdf}%
\end{picture}%
\setlength{\unitlength}{2368sp}%
\begingroup\makeatletter\ifx\SetFigFont\undefined%
\gdef\SetFigFont#1#2#3#4#5{%
  \reset@font\fontsize{#1}{#2pt}%
  \fontfamily{#3}\fontseries{#4}\fontshape{#5}%
  \selectfont}%
\fi\endgroup%
\begin{picture}(8908,2006)(-1303,-1964)
\put(6451,-1261){\makebox(0,0)[rb]{\smash{{\SetFigFont{10}{12.0}{\rmdefault}{\mddefault}{\updefault}{\color[rgb]{0,0,0}$v^s$}%
}}}}
\put(751,-961){\makebox(0,0)[lb]{\smash{{\SetFigFont{10}{12.0}{\rmdefault}{\mddefault}{\updefault}{\color[rgb]{0,0,0}$v^r$}%
}}}}
\put(2251,-1111){\makebox(0,0)[rb]{\smash{{\SetFigFont{10}{12.0}{\rmdefault}{\mddefault}{\updefault}{\color[rgb]{0,0,0}$v^s$}%
}}}}
\put(4951,-1261){\makebox(0,0)[lb]{\smash{{\SetFigFont{10}{12.0}{\rmdefault}{\mddefault}{\updefault}{\color[rgb]{0,0,0}$v^r$}%
}}}}
\end{picture}%

%% file: H.pdf_t
\begin{picture}(0,0)%
\includegraphics{H.pdf}%
\end{picture}%
\setlength{\unitlength}{2368sp}%
\begingroup\makeatletter\ifx\SetFigFont\undefined%
\gdef\SetFigFont#1#2#3#4#5{%
  \reset@font\fontsize{#1}{#2pt}%
  \fontfamily{#3}\fontseries{#4}\fontshape{#5}%
  \selectfont}%
\fi\endgroup%
\begin{picture}(2608,1684)(4997,-2264)
\put(7501,-811){\makebox(0,0)[lb]{\smash{{\SetFigFont{10}{12.0}{\rmdefault}{\mddefault}{\updefault}{\color[rgb]{0,0,0}$y$}%
}}}}
\put(5101,-2011){\makebox(0,0)[lb]{\smash{{\SetFigFont{10}{12.0}{\rmdefault}{\mddefault}{\updefault}{\color[rgb]{0,0,0}$w$}%
}}}}
\put(5101,-811){\makebox(0,0)[lb]{\smash{{\SetFigFont{10}{12.0}{\rmdefault}{\mddefault}{\updefault}{\color[rgb]{0,0,0}$x$}%
}}}}
\put(7501,-2011){\makebox(0,0)[lb]{\smash{{\SetFigFont{10}{12.0}{\rmdefault}{\mddefault}{\updefault}{\color[rgb]{0,0,0}$z$}%
}}}}
\end{picture}%

%% file: mu3-1.pdf_t
\begin{picture}(0,0)%
\includegraphics{mu3-1.pdf}%
\end{picture}%
\setlength{\unitlength}{2960sp}%
\begingroup\makeatletter\ifx\SetFigFont\undefined%
\gdef\SetFigFont#1#2#3#4#5{%
  \reset@font\fontsize{#1}{#2pt}%
  \fontfamily{#3}\fontseries{#4}\fontshape{#5}%
  \selectfont}%
\fi\endgroup%
\begin{picture}(4340,3369)(2426,-6598)
\put(6301,-5836){\makebox(0,0)[rb]{\smash{{\SetFigFont{10}{12.0}{\rmdefault}{\mddefault}{\updefault}{\color[rgb]{0.667,0.667,0.667}$m-1-d'$}%
}}}}
\put(3601,-5161){\makebox(0,0)[lb]{\smash{{\SetFigFont{10}{12.0}{\rmdefault}{\mddefault}{\updefault}{\color[rgb]{0.667,0.667,0.667}$n-2$}%
}}}}
\put(6751,-5761){\makebox(0,0)[lb]{\smash{{\SetFigFont{12}{14.4}{\rmdefault}{\mddefault}{\updefault}{\color[rgb]{0,0,0}$P_m$}%
}}}}
\put(4651,-4411){\makebox(0,0)[lb]{\smash{{\SetFigFont{12}{14.4}{\rmdefault}{\mddefault}{\updefault}{\color[rgb]{0,0,0}$x$}%
}}}}
\put(4426,-3436){\makebox(0,0)[lb]{\smash{{\SetFigFont{12}{14.4}{\rmdefault}{\mddefault}{\updefault}{\color[rgb]{0,0,0}$u$}%
}}}}
\put(4501,-6511){\makebox(0,0)[lb]{\smash{{\SetFigFont{12}{14.4}{\rmdefault}{\mddefault}{\updefault}{\color[rgb]{0,0,0}$v$}%
}}}}
\put(5701,-5311){\makebox(0,0)[lb]{\smash{{\SetFigFont{12}{14.4}{\rmdefault}{\mddefault}{\updefault}{\color[rgb]{0,0,0}$C'$}%
}}}}
\put(6151,-3886){\makebox(0,0)[lb]{\smash{{\SetFigFont{12}{14.4}{\rmdefault}{\mddefault}{\updefault}{\color[rgb]{0,0,0}$y$}%
}}}}
\put(2551,-5311){\makebox(0,0)[rb]{\smash{{\SetFigFont{12}{14.4}{\rmdefault}{\mddefault}{\updefault}{\color[rgb]{0,0,0}$P_l$}%
}}}}
\put(3001,-5611){\makebox(0,0)[lb]{\smash{{\SetFigFont{10}{12.0}{\rmdefault}{\mddefault}{\updefault}{\color[rgb]{0.667,0.667,0.667}$l-1-d$}%
}}}}
\put(3451,-4861){\makebox(0,0)[lb]{\smash{{\SetFigFont{12}{14.4}{\rmdefault}{\mddefault}{\updefault}{\color[rgb]{0,0,0}$C''$}%
}}}}
\put(2626,-3511){\makebox(0,0)[lb]{\smash{{\SetFigFont{12}{14.4}{\rmdefault}{\mddefault}{\updefault}{\color[rgb]{0,0,0}$z$}%
}}}}
\put(3601,-3736){\makebox(0,0)[lb]{\smash{{\SetFigFont{10}{12.0}{\rmdefault}{\mddefault}{\updefault}{\color[rgb]{0.667,0.667,0.667}$d$}%
}}}}
\put(5101,-3961){\makebox(0,0)[lb]{\smash{{\SetFigFont{10}{12.0}{\rmdefault}{\mddefault}{\updefault}{\color[rgb]{0.667,0.667,0.667}$d'$}%
}}}}
\put(4576,-5236){\makebox(0,0)[lb]{\smash{{\SetFigFont{12}{14.4}{\rmdefault}{\mddefault}{\updefault}{\color[rgb]{0,0,0}$P_n$}%
}}}}
\end{picture}%